\documentclass{amsart}
\usepackage{amssymb,latexsym}
\usepackage{amsfonts}
\usepackage{amsmath}
\usepackage{enumitem}

\newtheorem{theorem}{Theorem}[section]
\newtheorem{lemma}{Lemma}[section]

\newtheorem{prop}{Proposition}[section]

\numberwithin{equation}{section}

\newcommand{\ba}{\begin{array}}
\newcommand{\ea}{\end{array}}

\newcommand{\rank}{{\rm rank}}
\newcommand{\End}{{\rm End}}
\newcommand{\Id}{\textrm{Id}}

\newcommand{\Hom}{{\rm Hom}}

\def \qed{\cqfd}

\def\qed{\vbox{\hrule
\hbox{\vrule\hbox to 5pt{\vbox to 8pt{\vfil}\hfil}\vrule}\hrule}}

\newcommand{\beg}{\begin{eqnarray*}}
\newcommand{\begn}{\begin{eqnarray}}
\newcommand{\en}{\end{eqnarray*}}
\newcommand{\enn}{\end{eqnarray}}

\newcommand{\tr}{\mbox{\rm tr\,}}

\begin{document}
\title{The limit of the harmonic flow on flat complex vector bundle}
\author{Xi Zhang}
\address{Xi Zhang\\School of Mathematical Sciences\\
University of Science and Technology of China\\
Hefei, 230026,P.R. China\\ } \email{mathzx@ustc.edu.cn}
\thanks{The authors are partially supported by NSF in China No.11625106,
11801535 and 11721101. The research was partially supported by the project ``Analysis and Geometry on Bundle" of Ministry of Science and Technology of the People's Republic of China, No.SQ2020YFA070080.}

\begin{abstract}
In this paper we study the limiting behaviour of the harmonic flow on flat complex vector bundle, and prove the limit must be isomorphic to the graded flat complex vector bundle associated to the Jordan-H\"older filtration.
\end{abstract}

\maketitle


\section{Introduction}

Let $(E, D)$ be a flat complex vector bundle of rank $r$ over a compact Riemannian manifold $(M, g)$. We say $(E,D)$ is simple if it has no proper $D$-invariant sub-bundle and  $(E,D)$ is semi-simple if it is a direct sum of $D$-invariant sub-bundles. For the general case, there is a filtration of sub-bundles
\begin{equation}\label{HNS01}
0=E_{0}\subset E_{1}\subset \cdots \subset E_{i} \cdots \subset E_{l}=E ,
\end{equation}
such that every sub-bundle $E_{i}$ is $D$-invariant and every quotient bundle $(Q_{i}, D_{i}):=(E_{i}/E_{i-1}, D_{i})$ is flat and simple, which is called the Jordan-H\"older filtration of the flat complex vector bundle $(E, D)$.
It is well known that the above filtration may be not unique, but the following graded flat complex vector bundle
\begin{equation}Gr^{JH}(E, D
)=\oplus_{i=1}^{l}(Q_{i}, D_{i})
\end{equation} is unique in the sense of isomorphism. By the Riemann-Hilbert correspondence, we know that there is a one-to-one correspondence between the moduli space of fundamental group representations and the moduli space of flat vector bundles. If the flat bundle $(E, D)$ is corresponding to a representation $\tau : \pi_{1}(M) \rightarrow GL(r, \mathbb{C})$, the graded object $Gr^{JH}(E, D)$ is corresponding to the semi-simplification of $\tau $.

Given a Hermitian metric $H$ on $E$, there is a unique decomposition
\begin{equation}
D=D_{H}+\psi_{H},
\end{equation}
where $D_{H}$ is a unitary connection and $\psi_{H}\in \Omega^{1}(\mbox{End}(E))$ is self-adjoint with respect to $H$. A Hermitian metric $H$ is called harmonic on $(E, D)$ if it is a critical point of the energy functional $\int_{M}|\psi_{H}|^{2}dV_{g}$, i.e. it satisfies the Euler-Lagrange equation
  \begin{equation}
D_{H}^{\ast } \psi_{H}=0.
\end{equation}
Under the assumption that the flat complex vector bundle $(E, D)$ is semi-simple,  Corlette (\cite{Cor}) and Donaldson (\cite{Don3}) proved the existence of harmonic metric. Furthermore, when $(M, g)$ is a K\"ahler manifold, the existence of harmonic metric $H$ implies that there exists a poly-stable Higgs structure $ (D_{H}^{0, 1}, \psi_{H}^{1, 0} )$ on $E$. On the other hand, by the work of  Hitchin (\cite{H}) and Simpson (\cite{S2}) on Donaldson-Uhlenbeck-Yau theorem for Higgs bundles, one has the non-abelian Hodge correspondence, i.e. there is  an
equivalence of categories between the category of  poly-stable Higgs bundles with vanishing Chern numbers and the category of semi-simple flat bundles.

In order to obtain harmonic metrics, Corlette (\cite{Cor}) introduced the following heat flow
\begin{equation}\label{H1}
\frac{\partial \sigma (t)}{\partial t} \cdot (\sigma (t))^{-1}= D_{t, K}^{\ast }\psi_{t, K},
\end{equation}
where $K$ is a fixed Hermitian metric on $(E, D)$, $\sigma (t)\in \Gamma (\mbox{Aut}E)$ and
\begin{equation}D_{t}=\sigma (t) \{D\}=\sigma (t) \cdot D \cdot \sigma^{-1} (t)=D_{t, K} + \psi_{t, K}.\end{equation}
The heat flow (\ref{H1}) is equivalent to the following heat flow which involves flat connections,
\begin{equation}\label{H2}
\frac{\partial D_{t}}{\partial t} =-D_{t}\{D_{t, K}^{\ast }\psi_{t, K}\}.
\end{equation}
We call the above heat flow (\ref{H2}) the harmonic flow on the flat bundle $(E, D)$. Corlette proved the existence of long time solution $D_{t}$ ($0\leq t < \infty $) for the heat flow (\ref{H2}). By choosing a subsequence and taking suitable unitary gauge transformations, $D_{t_{i}}$ converges weakly to a flat connection $D_{\infty }$ in $L_{1}^{ p}$. Furthermore, if $(E, D)$ is simple, Corlette showed that the limit must lie in the complex gauge orbit of $D$, i.e. there exists $\eta_{\infty }\in \Gamma (\mbox{Aut}E)$ such that $D_{\infty } =\eta_{\infty } \cdot D \cdot \eta^{-1}_{\infty}$.

In this paper, we consider the limit of the harmonic flow (\ref{H2}) on the flat bundle $(E, D)$ which is not necessarily simple. Firstly, let's recall the limiting behaviour of the Yang-Mills flow on holomorphic vector bundles. For the Riemann surface case, Atiyah and Bott (\cite{AB}) pointed out that the limiting holomorphic bundle should be isomorphic to
the graded bundle associated to the Harder-Narasimhan-Seshadri filtration, and this conjecture has been proved by  Daskalopoulos (\cite{Da}). In \cite{BS}, Bando and Siu proposed an interesting question that the above Atiyah-Bott's conjecture should still hold for reflexive sheaf $\mathcal{E}$ over higher dimensional K\"ahler manifold.  When the sheaf $\mathcal{E}$ is locally free, this question was answered in the affirmative by Daskalopoulos and Wentworth (\cite{DW1}) for  K\"ahler surfaces case;  by Jacob (\cite{Ja3}) and Sibley (\cite{Sib})  for  higher dimensional case. The general reflexive sheaves case was  confirmed by Li, Zhang and the author (\cite{LZZ}). Inspired by this, it is natural to raise a question: should the limit of the harmonic flow (\ref{H2}) be isomorphic to the graded  flat complex vector bundle associated to the Jordan-H\"older filtration? When the base manifold $(M, g)$ is K\"ahler, this also is conjectured by Deng (\cite{De}) in his doctoral dissertation. In this paper, we solve this problem, i.e. we prove the following theorem.

\medskip

\begin{theorem}\label{thm0}
Let $(E, D)$ be a flat complex vector bundle over a compact Riemannian manifold $(M, g)$, and $D_{t}$ be the long time solution of the harmonic flow (\ref{H2}) with initial data $D$. Then the limiting flat bundle $(E, D_{\infty})$ must be isomorphic to the graded  flat complex vector bundle associated to the Jordan-H\"older filtration of $(E, D)$, i.e. we have:
\begin{equation}\label{I0}
(E, D_{\infty })\cong Gr^{JH}(E, D).
\end{equation}
\end{theorem}

\medskip

\medskip

This paper is organized as follows. In Section 2, we introduce some basic concepts and results about the harmonic flow on flat complex vector bundles. In section 3, we  give a proof of Theorem \ref{thm0}.
\medskip


\medskip

\section{Preliminaries}

Let $(M, g )$ be a compact Riemannian manifold of dimension $n$, $E$ be a   complex vector bundle over $M$ with  rank $r$.
Given any connection $D$ and Hermitian metric $H$ on $E$, there is a unique decomposition
\begin{equation}
D=D_{H}+\psi_{H},
\end{equation}
where $D_{H}$ is an $H$-unitary connection, $\psi_{H}\in \Omega^{1}(\mbox{End}(E))$ is $H$-self-adjoint, i.e. $\psi_{H}^{\ast H}=\psi_{H}$, and
\begin{equation}\label{ex1}
H(\psi_{H}X, Y)=\frac{1}{2}\{H(DX, Y)+H(X, DY)-dH(X, Y)\}
\end{equation}
for any $X, Y\in \Gamma (E)$. Suppose $K$ is another Hermitian metric on $E$, then we have
\begin{equation}\label{L1}
\begin{split}
\psi_{H}&=\frac{1}{2}h^{-1}\circ \psi_{K} \circ h + \frac{1}{2} \psi_{K} +\frac{1}{2} (D_{K}- h^{-1}\circ D_{K} \circ h)\\
&=h^{-1}\circ \psi_{K} \circ h + \frac{1}{2} (D- h^{-1}\circ D \circ h)\\
\end{split}
\end{equation}
and
\begin{equation}\label{L2}
\begin{split}
D_{H} &= \frac{1}{2} (\psi_{K} - h^{-1}\circ \psi_{K} \circ h +D_{K} + h^{-1}\circ D_{K} \circ h)\\
&=\psi_{H}- h^{-1}\circ \psi_{K} \circ h + h^{-1}\circ D_{K} \circ h \\
&=h^{-1}\circ D_{K} \circ h + \frac{1}{2} (D- h^{-1}\circ D \circ h).\\
\end{split}
\end{equation}
where $h=K^{-1}H$.

If $D$ is a flat connection, then
\begin{equation}\label{condition3}
0=F_{D}=D_{H}^{2}+\psi_{H}\wedge \psi_{H}+D_{H}\circ \psi_{H}+ \psi_{H}\circ D_{H}.
\end{equation}
Considering the self-adjoint and anti-self-adjoint parts of the above identity,  we obtain
\begin{equation}\label{co1}
D_{H}(\psi_{H})=0,
\end{equation}
and
\begin{equation}\label{D1}
D_{H}^{2}+\psi_{H}\wedge \psi_{H}=0.
\end{equation}

Let $H(t)$ be a family of Hermitian metrics on $E$. By direct computation, one can find that
\begin{equation}
\frac{\partial \psi_{H(t)}}{\partial t}=-\frac{1}{2} D_{H}(H^{-1}\frac{\partial H}{\partial t})+\frac{1}{2} \psi_{H}\circ H^{-1}\frac{\partial H}{\partial t}- \frac{1}{2} H^{-1}\frac{\partial H}{\partial t}\circ \psi_{H}.
\end{equation}
Choosing local coordinates $\{x^{i}\}_{i=1}^{n}$ on $M$, we  write $g=g_{ij}dx^{i}\otimes dx^{j}$, $\psi_{H}=(\psi_{H})_{k}dx^{k}$ and
\begin{equation}
|\psi_{H}|_{H}^{2}=g^{ij}\tr \{(\psi_{H})_{i}\circ (\psi_{H})_{j}^{\ast H}\},
\end{equation}
where $(\psi_{H})_{i} \in \Gamma (\End(E))$ and $(g^{ij})$ is the inverse matrix of $(g_{ij})$. After a straightforward calculation, one can check that
\begin{equation}
\begin{split}
\frac{\partial }{\partial t} |\psi_{H(t)}|_{H}^{2} &=2Re\langle\frac{\partial \psi_{H(t)}}{\partial t}-\frac{1}{2} \psi_{H}\circ H^{-1}\frac{\partial H}{\partial t}+\frac{1}{2} H^{-1}\frac{\partial H}{\partial t}\circ \psi_{H} , \psi_{H}\rangle_{H}\\
&=-Re\langle D_{H}( h^{-1}\frac{\partial h}{\partial t}) , \psi_{H}\rangle_{H}.\\
\end{split}
\end{equation}


Let $K$ be a fixed metric on $E$. Denote the group of smooth automorphisms of $E$ (which preserve the  metric $K$) as $\mathcal{G}$ ($\mathcal{U}_{K}$). Every $\sigma \in \mathcal{G}$ acts on the connection $D$ by
\begin{equation}
\sigma (D):= \sigma \circ D \circ \sigma^{-1} .
\end{equation}
For  $\sigma \in \mathcal{G}$, i.e. $H(X, Y)=K(\sigma X, \sigma Y)$ for any $X, Y \in \Gamma(E)$, set $H =K \sigma^{\ast K} \sigma $. One can see that
\begin{equation}
\begin{split}
K(\psi_{\sigma (D) , K} (X), Y)&= \frac{1}{2}\{K(\sigma (D)X, Y)+K(X, \sigma (D)Y)-dK(X, Y)\}\\
&= \frac{1}{2}\{H(D\circ \sigma^{-1} X, \sigma^{-1}Y)+H(\sigma^{-1}X, D\circ \sigma^{-1}Y)-dH(\sigma^{-1}X, \sigma^{-1}Y)\}\\
&= H(\psi_{D , H} \circ \sigma^{-1} (X), \sigma^{-1}Y)\\
&= K(\sigma \circ \psi_{D , H} \circ \sigma^{-1} (X), \sigma^{-1}Y).\\
\end{split}
\end{equation}
Then
\begin{equation}\label{L3}
\psi_{\sigma (D) , K} = \sigma \circ \psi_{D , H} \circ \sigma^{-1}
\end{equation}
and
\begin{equation}\label{L4}
\sigma (D)_{K} = \sigma \circ D_{H} \circ \sigma^{-1} .
\end{equation}
By (\ref{L1}), (\ref{L2}), (\ref{L3}) and (\ref{L4}), we have (or By the definition, we have!!!!!!)
\begin{equation}\label{L5}
\psi_{\sigma (D) , K} = (\sigma^{\ast K})^{-1} \circ \psi_{D , K} \circ \sigma^{\ast K}+\frac{1}{2} \sigma \circ D \circ \sigma^{-1}-\frac{1}{2} (\sigma^{\ast K})^{-1} \circ D \circ \sigma^{\ast K}
\end{equation}
and
\begin{equation}\label{L6}
\sigma (D)_{K} = (\sigma^{\ast K})^{-1} \circ D_{K} \circ \sigma^{\ast K}+\frac{1}{2} \sigma \circ D \circ \sigma^{-1}-\frac{1}{2} (\sigma^{\ast K})^{-1} \circ D \circ \sigma^{\ast K}.
\end{equation}
Let $h=\sigma^{\ast K}\circ \sigma $. There holds that
\begin{equation}\label{L7}
\psi_{\sigma (D) , K} = (\sigma^{\ast K})^{-1} \circ (\psi_{D , K}-\frac{1}{2}D(h)\circ h^{-1} ) \circ \sigma^{\ast K}
\end{equation}
and
\begin{equation}\label{L8}
\sigma (D)_{K} = (\sigma^{\ast K})^{-1} \circ (D_{K}-\frac{1}{2}D(h)\circ h^{-1} )  \circ \sigma^{\ast K}.
\end{equation}
From the definition, it  is easy to see that
\begin{equation}
\langle\varphi_{1}, \varphi_{2}\rangle_{H}=\langle\sigma \circ \varphi_{1}\circ \sigma^{-1}, \sigma \circ \varphi_{2}\circ \sigma^{-1}\rangle_{K}
\end{equation}
for any $\varphi_{1}, \varphi_{2} \in \Gamma(\End(E))$. We know
\begin{equation}
\begin{split}
\langle\psi_{D, H}, D_{H}\varphi \rangle_{H}&=\langle\psi_{D, H}, (\sigma^{-1} \circ (\sigma (D))_{K} \circ \sigma )(\varphi ) \rangle_{H}\\
&=\langle\psi_{D, H}, \sigma^{-1} \circ \{(\sigma (D))_{K}(\sigma \circ \varphi \circ \sigma^{-1})\} \circ \sigma  \rangle_{H}\\
&=\langle\sigma \circ \psi_{D, H}\circ \sigma^{-1},  (\sigma (D))_{K}(\sigma \circ \varphi \circ \sigma^{-1})  \rangle_{K}\\
&=\langle \psi_{\sigma(D), K},  (\sigma (D))_{K}(\sigma \circ \varphi \circ \sigma^{-1})  \rangle_{K}\\
\end{split}
\end{equation}
and then
\begin{equation}\label{C120}
(\sigma (D))_{K}^{\ast }\psi_{\sigma(D), K}=\sigma \circ D_{H}^{\ast }\psi_{D, H} \circ \sigma^{-1}.
\end{equation}
On the other hand, one can check that $(\sigma (D))_{K}^{\ast }\psi_{\sigma(D), K}$
is self-adjoint with respect to the metric $K$.

\begin{lemma}\label{C0}
Let $(E, D)$ be a flat complex vector bundle on a compact Riemannian manifold $(M, g)$, and $K$ be a Hermitian metric on $E$. For any $\sigma \in \mathcal{G}$, we have
\begin{equation}\label{C01}
\langle\sigma^{-1}\circ (\sigma (D))_{K}^{\ast }\psi_{\sigma(D), K}\circ \sigma -D_{K}^{\ast }\psi_{D, K}, h\rangle_{K}=\frac{1}{2}\Delta \tr h -\frac{1}{2}\langle D(h) \circ h^{-1}, D(h)\rangle_{K},
\end{equation}
\begin{equation}\label{C012}
\langle D_{K}^{\ast }\psi_{D, K}-\sigma^{-1}\circ (\sigma (D))_{K}^{\ast }\psi_{\sigma(D), K}\circ \sigma , h^{-1}\rangle_{K}=\frac{1}{2}\Delta \tr h^{-1} -\frac{1}{2}\langle h\circ D(h^{-1}), D(h^{-1})\rangle_{K}
\end{equation}
and
\begin{equation}\label{C02}
\langle\sigma^{-1}\circ (\sigma (D))_{K}^{\ast }\psi_{\sigma(D), K}\circ \sigma -D_{K}^{\ast }\psi_{D, K}, s\rangle_{K}=\frac{1}{4}\Delta |s|_{K}^{2} -\frac{1}{2}\langle D(h) \circ h^{-1}, D(s)\rangle_{K},
\end{equation}
where $h=\sigma^{\ast K}\circ \sigma $ and $s=\log h$.
\end{lemma}

\begin{proof}By (\ref{L7}) and (\ref{L8}), and choosing local normal coordinates $\{x^{i}\}_{i=1}^{n}$ centered at the considered point, one can easily check that
\begin{equation}\label{C130}
\begin{split}
&\sigma^{-1}\circ (\sigma (D))_{K}^{\ast }\psi_{\sigma(D), K}\circ \sigma =h^{-1}\circ (D_{K}^{\ast }\psi_{D, K} -\frac{1}{2}D_{K}^{\ast }(D(h) \circ h^{-1}))\circ h \\
&+\frac{1}{2}g^{ij}h^{-1}\circ (D_{\frac{\partial }{\partial x^{i}}}(h) \circ h^{-1} \circ \psi_{D, K}(\frac{\partial }{\partial x^{j}})-\psi_{D, K}(\frac{\partial }{\partial x^{j}}) \circ D_{\frac{\partial }{\partial x^{i}}}(h) \circ h^{-1})\circ h ,
\end{split}
\end{equation}
and then
\begin{equation}\label{C021}
\begin{split}
&\langle \sigma^{-1}\circ (\sigma (D))_{K}^{\ast }\psi_{\sigma(D), K}\circ \sigma -D_{K}^{\ast }\psi_{D, K}, s\rangle_{K}\\
=&\langle -\frac{1}{2}D_{K}^{\ast }(D(h) \circ h^{-1}), h^{-1}\circ s \circ h \rangle_{K}\\
&+ \langle \frac{1}{2}g^{ij}\circ (D_{\frac{\partial }{\partial x^{i}}}(h) \circ h^{-1} \circ \psi_{D, K}(\frac{\partial }{\partial x^{j}})-\psi_{D, K}(\frac{\partial }{\partial x^{j}}) \circ D_{\frac{\partial }{\partial x^{i}}}(h) \circ h^{-1}), h^{-1}\circ s \circ h \rangle_{K}\\
=& -\frac{1}{2}\tr (D_{K}^{\ast }(D(h) \circ h^{-1})\circ s)\\
&+ \frac{1}{2}g^{ij}\tr (D_{\frac{\partial }{\partial x^{i}}}(h) \circ h^{-1} \circ (\psi_{D, K}(\frac{\partial }{\partial x^{j}})\circ s -s\circ \psi_{D, K}(\frac{\partial }{\partial x^{j}})))\\
=& \frac{1}{2}g^{ij}\frac{\partial }{\partial x^{j}}\tr (D_{\frac{\partial }{\partial x^{i}}}(h) \circ h^{-1}\circ s)\\
&-\frac{1}{2}g^{ij}\tr (D_{\frac{\partial }{\partial x^{i}}}(h) \circ h^{-1} \circ (D_{K, \frac{\partial }{\partial x^{j}}}(s)-\psi_{D, K}(\frac{\partial }{\partial x^{j}})\circ s +s\circ \psi_{D, K}(\frac{\partial }{\partial x^{j}})))\\
=&\frac{1}{4}\Delta |s|_{K}^{2} -\frac{1}{2}\langle D(h) \circ h^{-1}, D(s)\rangle_{K}.
\end{split}
\end{equation}
Here we have used the following identity
\begin{equation}
\tr (D_{\frac{\partial }{\partial x^{i}}}(h) \circ h^{-1}\circ s)=\tr (s\circ D_{\frac{\partial }{\partial x^{i}}}s).
\end{equation}
Immediately (\ref{C01}) and (\ref{C012}) can be proved in a similar way.
\end{proof}

\begin{lemma}\label{lem:x1}
Let $(E, \hat{D})$ be a flat complex vector bundle on a compact Riemannian manifold $(M, g)$, and $K$ be a Hermitian metric on $E$. Assume $\hat{D}_{K}^{\ast }\psi_{\hat{D}, K}=0$, then
$(E, \hat{D})$ must be semi-simple.
\end{lemma}
\begin{proof}
Suppose that $(E, \hat{D})$ is not simple, and choose a $\hat{D}$-invariant  sub-bundle $S$ which is minimal rank. Then there exists an exact sequence
\begin{equation}
0\rightarrow S\rightarrow E \rightarrow Q \rightarrow 0.
\end{equation}
Denote $D_{S}$ and $D_{Q}$ (respectively, $K_{S}$ and $K_{Q}$) the connections (respectively, metrics) on the sub-bundle $S$ and the quotient bundle $Q$ induced by the connection $\hat{D}$ (respectively, metric $K$). For the Hermitian metric $K$ on $E$, we have the following bundle isomorphism
\begin{equation}\label{is1}
f_{K}: S\oplus Q \rightarrow E , \qquad (X, [Y])\mapsto i(X) +(\Id_{E}-\pi _{K})(Y),
\end{equation}
 where $X\in \Gamma(S)$, $Y\in \Gamma(E) $, $i:S\hookrightarrow E$ is the inclusion and $\pi_{K}: E\rightarrow E$ is the orthogonal projection into $S$ with respect to the metric $K$. Since $S$ is $\hat{D}$-invariant, we know
 \begin{equation}\label{pi1}
 \pi_{K}=(\pi_{K})^{2}=(\pi_{K})^{\ast K}
 \end{equation}
and
 \begin{equation}\label{pi2}
 (\Id_{E}-\pi_{K})\circ \hat{D}(\pi_{K})=0.
 \end{equation}

 By the definition, the pulling back metric is
\begin{equation}
f_{K}^{\ast}(K)=
\begin{pmatrix}
K_{S} &  0 \\
0   & K_{Q}\\
\end{pmatrix},
\end{equation}
and the pulling back connection is
\begin{equation}
f_{K}^{\ast}(\hat{D})=\begin{pmatrix}
D_{S} &  \beta  \\
0   & D_{Q}\\
\end{pmatrix},
\end{equation}
where $\beta \in \Omega ^{1}({\rm Hom} (Q, S))$  will be called  the  second fundamental form. One can check that
\begin{equation}\label{pi3}
\beta ([Y])=-\pi_{K}\circ (\hat{D}\pi_{K})(Y),
\end{equation}
where $Y\in \Gamma(E)$.
Because $\hat{D}$ is flat, we have
\begin{equation}
D_{S}^{2}=0, \quad D_{Q}^{2}=0, \quad D_{S}\circ \beta +\beta \circ D_{Q}=0.
\end{equation}
It is easy to see that
\begin{equation}\label{Q9}
f_{K}^{\ast }(\psi_{\hat{D}, K})=\begin{pmatrix}
\psi_{D_{S}, K_{S}} &  \frac{1}{2}\beta  \\
\frac{1}{2}\beta^{\ast }   & \psi_{D_{Q}, K_{Q}}\\
\end{pmatrix},
\end{equation}

\begin{equation}
f_{K}^{\ast }(\hat{D}_{ K})=\begin{pmatrix}
D_{K_{S}} &  \frac{1}{2}\beta  \\
-\frac{1}{2}\beta^{\ast }   & D_{K_{Q}}\\
\end{pmatrix},
\end{equation}
where $\beta^{\ast}\in \Omega ^{1}(\Hom (Q, S))$ is the adjoint of $\beta $ with respect to the metrics $K_{S}$ and $K_{Q}$.
In the following,  we choose the normal coordinates $\{x^{i}\}_{i=1}^{n}$ centered at the considered point $p\in M$. A direct calculation yields
\begin{equation}\label{codazzi1}
\begin{split}
&f_{K}^{-1}\circ \hat{D}_{K}^{\ast} \psi_{\hat{D}, K} \circ f_{K}= f_{K}^{\ast }(\hat{D}_{ K})^{\ast }\{f_{K}^{\ast }(\psi_{\hat{D}, K})\}\\&=\begin{pmatrix}
D_{K_{S}}^{\ast}\psi_{D_{S}, K_{S}} -\frac{1}{2}g^{ij}\beta_{i}\circ \beta_{j}^{\ast} &  \frac{1}{2}D_{K, Q^{\ast}\otimes S}^{\ast }\beta -\frac{1}{2}g^{ij}(\beta_{i}\circ \psi_{Q, j}-\psi_{S, j}\circ \beta_{i}) \\
\frac{1}{2}D_{K, S^{\ast}\otimes Q}^{\ast }\beta^{\ast } +\frac{1}{2}g^{ij}(\beta_{i}^{\ast}\circ \psi_{S, j}-\psi_{Q, j}\circ \beta_{i}^{\ast})   & D_{K_{Q}}^{\ast}\psi_{D_{Q}, K_{Q}} +\frac{1}{2}g^{ij}\beta_{i}^{\ast }\circ \beta_{j}\\
\end{pmatrix},
\end{split}
\end{equation}
where $\beta_{i}=\beta (\frac{\partial }{\partial x^{i}})$, $\beta_{j}^{\ast}=\beta^{\ast} (\frac{\partial }{\partial x^{j}})$, $\psi_{S, i}=\psi_{S}(\frac{\partial }{\partial x^{i}})$ and $\psi_{Q, j}=\psi_{Q}(\frac{\partial }{\partial x^{j}})$. Due to $\hat{D}_{K}^{\ast} \psi_{\hat{D}, K}=0$, (\ref{codazzi1}) implies
\begin{equation}
D_{K_{S}}^{\ast}\psi_{D_{S}, K_{S}} -\frac{1}{2}g^{ij}\beta_{i}\circ \beta_{j}^{\ast}=0,
\end{equation}
and
\begin{equation}
\int_{M}|\beta |^{2}dV_{g}=2\int_{M}\langle D_{K_{S}}^{\ast}\psi_{D_{S}, K_{S}}, \Id_{S}\rangle_{K_{S}}dV_{g}=0.
\end{equation}
So $(E, \hat{D})\cong (S, D_{S})\oplus (Q, D_{Q})$, where $(S, D_{S})$ is a simple flat bundle and $(Q, D_{Q})$ is a flat bundle with $D_{K_{Q}}^{\ast}\psi_{D_{Q}, K_{Q}}=0$. Applying the above argument to $(Q, D_{Q})$, we obtain an isomorphism
\begin{equation}
(E, \hat{D})\cong \oplus_{i=1}^{l} (Q_{i}, D_{Q_{i}}),
\end{equation}
where every $(Q_{i}, D_{Q_{i}})$ is a simple flat bundle.
\end{proof}

\medskip

Let $\sigma (t)$ be a solution of the heat flow (\ref{H1}), i.e. it satisfies
\begin{equation}
\frac{\partial \sigma(t) }{\partial t}\circ \sigma^{-1}(t) = (\sigma(t) \{D\})_{K}^{\ast }\psi_{\sigma (t) \{D\}, K},
\end{equation}
then
\begin{equation}
\begin{split}
\frac{\partial }{\partial t}\sigma(t) \{D\}&= -\sigma(t) \{D\} (\frac{\partial \sigma (t)}{\partial t}\circ \sigma^{-1}(t))\\
&=-\sigma(t) \{D\} ((\sigma(t) \{D\})_{K}^{\ast }\psi_{\sigma (t) \{D\}, K}).\\
\end{split}
\end{equation}
Considering the self-adjoint and anti-self-adjoint parts of the above identity,  we have
\begin{equation}
\frac{\partial }{\partial t}\sigma(t) \{D\}_{K}
=-[\psi_{\sigma (t) \{D\}, K}, (\sigma(t) \{D\})_{K}^{\ast }\psi_{\sigma (t) \{D\}, K}],
\end{equation}
and
\begin{equation}
\frac{\partial }{\partial t}\psi_{\sigma (t) \{D\}, K}
=-\sigma(t) \{D\}_{K} ((\sigma(t) \{D\})_{K}^{\ast }\psi_{\sigma (t) \{D\}, K}).
\end{equation}

Let's recall some basic estimates of the heat flows (\ref{H1}) and (\ref{H2}).

\begin{lemma}\label{b1}{\bf(\cite{Cor})}
Let $(E, D)$ be a flat complex vector bundle on a compact Riemannian manifold $(M, g)$, and $K$ be a Hermitian metric on $E$. If  $\sigma (t)$ is a solution of the heat flow (\ref{H1}), then we have
\begin{equation}\label{C11}
\frac{d}{dt}\|\psi_{\sigma (t) \{D\}, K}\|_{L^2}^{2}=-2\|(\sigma(t) \{D\})_{K}^{\ast }\psi_{\sigma (t) \{D\}, K}\|_{L^2}^{2},
\end{equation}
\begin{equation}\label{C12}
\begin{split}
&(\Delta -\frac{\partial }{\partial t})|\psi_{\sigma (t) \{D\}, K}|_{K}^{2}
=2|\nabla^{(\sigma(t) \{D\})_{K}}\psi_{\sigma (t) \{D\}, K}|_{K}^{2}\\ &+2\langle\psi_{\sigma (t) \{D\}, K}\circ Ric , \psi_{\sigma (t) \{D\}, K}\rangle_{K} +2 |[\psi_{\sigma (t) \{D\}, K}, \psi_{\sigma (t) \{D\}, K}]|_{K}^{2}
\end{split}
\end{equation}
and
\begin{equation}\label{C13}
\begin{split}
&(\Delta -\frac{\partial }{\partial t})|(\sigma(t) \{D\})_{K}^{\ast }\psi_{\sigma (t) \{D\}, K}|_{K}^{2}\\
&= 2|(\sigma(t) \{D\})_{K}((\sigma(t) \{D\})_{K}^{\ast }\psi_{\sigma (t) \{D\}, K})|_{K}^{2}+2|[ \psi_{\sigma (t) \{D\}, K}, (\sigma(t) \{D\})_{K}^{\ast }\psi_{\sigma (t) \{D\}, K}]|_{K}^{2}.
\end{split}
\end{equation}
\end{lemma}

\begin{prop}\label{b12}{\bf(\cite{Cor})}
Let $(E, D)$ be a flat complex vector bundle on a compact Riemannian manifold $(M, g)$, and $K$ be a Hermitian metric on $E$. The harmonic flow (\ref{H1}) has a long time solution $\sigma (t)$ for $t\in [0, \infty )$. Furthermore, for every sequence $t_{i}\rightarrow \infty $ there exists
a subsequence $t_j$ such that $\sigma(t_{j}) \{D\}=(\sigma(t_{j}) \{D\})_{K}+ \psi_{\sigma (t_{j}) \{D\}, K}$ converges weakly,  modulo $K$-unitary gauge transformations, to a flat connection $D_{\infty}=D_{\infty , K}+ \psi_{\infty ,K}$ in $L_{1}^{p}$-topology,  and $D_{\infty ,K}^{\ast}\psi_{\infty , K}=0$.
\end{prop}

\medskip

In the following, we will show that the above convergence can be strengthened to in $C^{\infty}$-topology.
Using (\ref{C12}), we deduce
\begin{equation}
(\Delta -\frac{\partial }{\partial t})|\psi_{\sigma (t) \{D\}, K}|_{K}^{2}\geq -C_{1}|\psi_{\sigma (t) \{D\}, K}|_{K}^{2},
\end{equation}
equivalently
\begin{equation}\label{C14}
(\Delta -\frac{\partial }{\partial t})(e^{-C_{1}t}|\psi_{\sigma (t) \{D\}, K}|_{K}^{2})\geq 0,
\end{equation}
where $C_{1}$ is a positive constant depending only on the Ricci curvature of $(M, g)$. Let $f(x, t)=\int_{M} \chi (x, y, t-t_{0})e^{-C_{1}t_{0}}|\psi_{\sigma (t_{0}) \{D\}, K}|_{K}^{2}(y) d V_{g} (y)$, where $\chi $ is the heat kernel of $(M, g)$. Of course (\ref{C14}) implies:
\begin{equation}
(\Delta -\frac{\partial }{\partial t} ) (e^{-C_{1}t}|\psi_{\sigma (t) \{D\}, K}|_{K}^{2} -f(x, t))\geq 0
\end{equation}
and
\begin{equation}
f(\cdot , t_{0})=e^{-C_{1}t_{0}}|\psi_{\sigma (t_{0}) \{D\}, K}|_{K}^{2}.
\end{equation}
From the maximum principle and (\ref{C11}), for any $t_{0}\geq 0$,  it follows that
\begin{equation}
\begin{split}
\max _{M} e^{-C_{1}(t_{0}+1)}|\psi_{\sigma (t_{0}+1) \{D\}, K}|_{K}^{2}\leq & \max _{M} f(x, t_{0}+1)\\
=& \int_{M} \chi (x, y, 1)e^{-C_{1}t_{0}}|\psi_{\sigma (t_{0}) \{D\}, K}|_{K}^{2}(y) d V_{g} (y)\\
\leq & C_{2}e^{-C_{1}t_{0}}\int_{M} |\psi_{\sigma (t_{0}) \{D\}, K}|_{K}^{2} dV_{g} \\
\leq & C_{2}e^{-C_{1}t_{0}}\int_{M} |\psi_{D, K}|_{K}^{2} dV_{g} ,\\
\end{split}
\end{equation}
and then
\begin{equation}\label{C15}
\max _{M} |\psi_{\sigma (t_{0}+1) \{D\}, K}|_{K}^{2}\leq C_{2}e^{C_{1}}\int_{M} |\psi_{D, K}|_{K}^{2} dV_{g},
\end{equation}
where $C_{2}$ is a positive constant depending only on the upper bound of $\chi (x, y, 1)$. Hence we know that $\sup_{M}|\psi_{\sigma (t)\{D\}, K}|_{K}$ is uniformly bounded.

Choosing local normal coordinates $\{x^{i}\}_{i=1}^{n}$ centered at the considered point, we have
\begin{equation}\label{h11}
\begin{split}
&\Delta |\nabla^{\sigma }\psi_{\sigma (t) \{D\}, K}|_{K}^{2}=  2 |\nabla^{\sigma }\nabla^{\sigma }\psi_{\sigma (t) \{D\}, K}|_{K}^{2}\\
&+2Re\{g^{ij}\langle \nabla^{\sigma }_{\frac{\partial }{\partial x^{i}}} \nabla^{\sigma }_{\frac{\partial }{\partial x^{i}}} \nabla^{\sigma }\psi_{\sigma (t) \{D\}, K}, \nabla^{\sigma }\psi_{\sigma (t) \{D\}, K} \rangle \},
\end{split}
\end{equation}
where $\nabla^{\sigma }$ is the covariant derivative induced by the connection $(\sigma(t) \{D\})_{K}$ and the Levi-Civita connection $\nabla $ of $(M, g)$. Denote
\begin{equation}
\nabla^{\sigma }\psi_{\sigma (t) \{D\}, K}=\psi_{m, l}dx^{m}\otimes dx^{l},
\end{equation}
and we have
\begin{equation}
g^{ij} \nabla^{\sigma }_{\frac{\partial }{\partial x^{i}}} \nabla^{\sigma }_{\frac{\partial }{\partial x^{i}}} \nabla^{\sigma }\psi_{\sigma (t) \{D\}, K}=g^{ij}\psi_{m, lji}dx^{m}\otimes dx^{l},
\end{equation}
where $\psi_{m, lji}$ denotes the component of the covariant derivative. According to the Ricci identity, one can check that
\begin{equation}\label{h12}
\begin{split}
g^{ij}\psi_{m, lji}=& g^{ij} (\psi_{j, iml} +\psi_{a, l}R^{a}_{jmi}+\psi_{a}R^{a}_{jmi, l}+\psi_{a, i}R^{a}_{mlj}+\psi_{a}R^{a}_{mlj,i}++\psi_{a, m}R^{a}_{jli}+\psi_{j,a}R^{a}_{mli}\\
&+[F_{im,l}, \psi_{j}]+[F_{im}, \psi_{j, l}]+[F_{il}, \psi_{j, m}]+[F_{jl,i}, \psi_{m}]+[F_{jl}, \psi_{m, i}]),
\end{split}
\end{equation}
where $R$ denotes the Riemannian curvature of $g$ and $F$ denotes the curvature of the connection $(\sigma(t) \{D\})_{K}$. By the harmonic flow (\ref{H1}), we have
\begin{equation}\label{h13}
\begin{split}
&\frac{\partial }{\partial t}|\nabla^{\sigma }\psi_{\sigma (t) \{D\}, K}|_{K}^{2} =-2Re\langle \nabla^{\sigma } ((\sigma(t) \{D\})_{K}((\sigma(t) \{D\})_{K}^{\ast }\psi_{\sigma (t) \{D\}, K})) , \nabla^{\sigma }\psi_{\sigma (t) \{D\}, K}\rangle \\
&-2Re \langle [[\psi_{\sigma (t) \{D\}, K}, (\sigma(t) \{D\})_{K}^{\ast }\psi_{\sigma (t) \{D\}, K}], \psi_{\sigma (t) \{D\}, K}] , \nabla^{\sigma }\psi_{\sigma (t) \{D\}, K}\rangle .
\end{split}
\end{equation}
On the other hand, it is straightforward to check that
\begin{equation}\label{h14}
\nabla^{\sigma } ((\sigma(t) \{D\})_{K}((\sigma(t) \{D\})_{K}^{\ast }\psi_{\sigma (t) \{D\}, K}))=-g^{ij}\psi_{j,iml}dx^{m}\otimes dx^{l}.
\end{equation}
From (\ref{h11}), (\ref{h12}), (\ref{h13}) and (\ref{h14}), we get
\begin{equation}\label{h1}
\begin{split}
&(\Delta -\frac{\partial }{\partial t})|\nabla^{\sigma }\psi_{\sigma (t) \{D\}, K}|_{K}^{2}
\geq 2|\nabla^{\sigma }\nabla^{\sigma }\psi_{\sigma (t) \{D\}, K}|_{K}^{2}\\ &-C_{3}(|Rm|+|\psi_{\sigma (t) \{D\}, K}|_{K}^{2})|\nabla^{\sigma }\psi_{\sigma (t) \{D\}, K}|_{K}^{2}\\ & -C_{4}(|\psi_{\sigma (t) \{D\}, K}|_{K}^{2}|(\sigma(t) \{D\})_{K}^{\ast }\psi_{\sigma (t) \{D\}, K}|_{K}+|\psi_{\sigma (t) \{D\}, K}|_{K}|\nabla Rm|)|\nabla^{\sigma }\psi_{\sigma (t) \{D\}, K}|_{K},\\
\end{split}
\end{equation}
where  $C_{3}$ and $C_{4}$ are uniform constants depending only on $\dim(M)$ and $\rank (E)$.  (\ref{C12}) and (\ref{C15}) yields
\begin{equation}
\int_{M\times [t_{0}, t_{0}+3 ]}|\nabla^{\sigma }\psi_{\sigma (t) \{D\}, K}|_{K}^{2} dV_{g} dt\leq C_{5}
\end{equation}
for all $t_{0}\geq 0$, where $C_{3}$ is a uniform constant. Making use of (\ref{h1}) and the Moser's parabolic estimate, one can derive
\begin{equation}
\sup_{M\times [t_{0}+1, t_{0}+2 ]}|\nabla^{\sigma }\psi_{\sigma (t) \{D\}, K}|_{K}^{2} \leq C_{6}.
\end{equation}
Furthermore, it holds that
\begin{equation}\label{h2}
\begin{split}
&(\Delta -\frac{\partial }{\partial t})|(\nabla^{\sigma })^{\alpha }\psi_{\sigma (t) \{D\}, K}|_{K}^{2}
\geq 2|(\nabla^{\sigma })^{\alpha +1}\psi_{\sigma (t) \{D\}, K}|_{K}^{2}\\&-C_{5}|(\nabla^{\sigma })^{\alpha }\psi_{\sigma (t) \{D\}, K}|_{K}^{2}-C_{6}|(\nabla^{\sigma })^{\alpha }\psi_{\sigma (t) \{D\}, K}|_{K},
\end{split}
\end{equation}
where $C_{5}$ and $C_{6}$ are uniform positive constants depending only on $|Rm|, \quad |\nabla Rm|$, $\cdots $, $|(\nabla )^{\alpha }Rm |$, and $|\psi_{\sigma (t) \{D\}, K}|_{K}, \cdots , |(\nabla^{\sigma })^{\alpha -1 }\psi_{\sigma (t) \{D\}, K}|_{K}$. Using (\ref{D1}), (\ref{h2}) and repeating the above argument, we have the following uniform $C^{\infty }$-estimates,
\begin{equation}
\sup_{[0, \infty ) \times M}(|(\nabla^{\sigma })^{\alpha }F_{(\sigma(t) \{D\})_{K}}|^{2}+|(\nabla^{\sigma })^{\alpha }\psi_{\sigma (t) \{D\}, K}|_{K}^{2})\leq \hat{C}_{\alpha }
\end{equation}
for every $0\leq \alpha <\infty $. Then, applying a result of Donaldson-Kronheimer(Theorem 2.3.7 in \cite{DK}) and  Hong-Tian's argument (Proposition 6 in \cite{HT}), we obtain the following proposition.

\begin{prop}\label{pp:0}
Let $(E, D)$ be a flat complex vector bundle over a compact Riemannian manifold $(M, g)$, and $K$ be a Hermitian metric on $E$. If  $\sigma (t)$ is a longtime solution of the heat flow (\ref{H1}), then for every sequence $t_{i}\rightarrow \infty $ there exists
a subsequence $t_j$ such that $\sigma(t_{j}) \{D\}=(\sigma(t_{j}) \{D\})_{K}+ \psi_{\sigma (t_{j}) \{D\}, K}$ converges,  modulo $K$-unitary gauge transformations, to a flat connection $D_{\infty}=D_{K, \infty}+ \psi_{\infty}$ in $C^{\infty}$-topology,  and $D_{K, \infty}^{\ast}\psi_{\infty}=0$.
\end{prop}

\medskip


\section{Proof of Theorem \ref{thm0}}

 The following proposition about the existence of the Jordan-H\"{o}lder filtration and the uniqueness of the  graded flat complex vector bundle should be well known to experts, and we give a proof here for the reader's convenience.

\begin{prop}\label{JH}
Let $(E, D)$ be a flat complex vector bundle. There is a filtration of sub-bundles
\begin{equation}\label{HNS010}
0=E_{0}\subset E_{1}\subset \cdots \subset E_{i} \cdots \subset E_{l}=E ,
\end{equation}
such that every sub-bundle $E_{i}$ is $D$-invariant and every quotient bundle $(Q_{i}, D_{Q_{i}}):=(E_{i}/E_{i-1},  D_{Q_{i}})$ is flat and simple. Furthermore, the  graded flat complex vector bundle
$\oplus_{i=1}^{l}(Q_{i}, D_{Q_{i}})
$ is unique in the sense of isomorphism.
\end{prop}

\begin{proof}
Suppose that $(E, D)$ is not simple. Let $E_{1}$ be a $D$-invariant sub-bundle of $E$ with minimal rank. Then we have the following exact sequence
\begin{equation}
0\rightarrow E_{1}\xrightarrow{i_{0}} E \xrightarrow{P} Q \rightarrow 0,
\end{equation}
 and $(E_{1}, D_{E_{1}})$ is simple. By induction, we can assume that there is a Jordan-H\"{o}lder filtration of the flat bundle $(Q, D_{Q})$, i.e.
 \begin{equation}\label{HNS0100}
0=\hat{Q}_{0}\subset \hat{Q}_{1}\subset \cdots \subset \hat{Q}_{i} \cdots \subset \hat{Q}_{l-1}=Q .
\end{equation}
Choosing a Hermitian metric $K$ on $E$, we get a bundle isomorphism $P^{\ast K}: Q\rightarrow E_{1}^{\bot }$ , where $P^{\ast K}$ is the adjoint of the projection $P$ with respect to the metric $K$. Set
\begin{equation}
E_{i}=E_{1}\oplus P^{\ast K}(\hat{Q}_{i-1})
\end{equation}
for all $1<i \leq l$. One can easily  find that every $E_{i}$ is $D$-invariant and $(E_{i}/E_{i-1},  D_{Q_{i}})$ is simple. So, we obtain a Jordan-H\"{o}lder filtration of $(E, D)$.

Suppose that there is another Jordan-H\"{o}lder filtration of $(E, D)$,
\begin{equation}\label{HNS01000}
0=\tilde{E}_{0}\subset \tilde{E}_{1}\subset \cdots \subset \tilde{E}_{i} \cdots \subset \tilde{E}_{\tilde{l}}=E.
\end{equation}
We will show that
\begin{equation}\oplus_{i=1}^{l}(Q_{i}, D_{Q_{i}})\cong \oplus_{\alpha =1}^{\tilde{l}}(\tilde{Q}_{\alpha }, D_{\tilde{Q}_{\alpha }}).
\end{equation}
It is straightforward to check that there exist bundle isomorphisms  $f:\oplus_{i=1}^{l}Q_{i}\rightarrow E$ and  $\tilde{f}:\oplus_{\alpha =1}^{\tilde{l}}\tilde{Q}_{\alpha }\rightarrow E$ such that
\begin{equation}\label{311}
\begin{split}
f^{*}(D)=\left(\begin{array}{ccccc}
  D_{Q_{1}} &\cdots & \beta_{1l} \\
   \vdots  &\ddots & \vdots\\
  0  & \cdots & D_{Q_{l}}
\end{array}\right)
\end{split}
\end{equation}
and
\begin{equation}\label{312}
\begin{split}
\tilde{f}^{*}(D)=\left(\begin{array}{ccccc}
  D_{\tilde{Q}_{1}} &\cdots & \gamma_{1\tilde{l}} \\
   \vdots  &\ddots & \vdots\\
  0  & \cdots & D_{\tilde{Q}_{\tilde{l}}}
\end{array}\right).
\end{split}
\end{equation}
Let $\eta =f^{-1}\circ \tilde{f}=(\eta_{i\alpha })$, where $\eta_{i\alpha }\in \Hom(\tilde{Q}_{j}, Q_{i})$. (\ref{311}) and (\ref{312}) mean
\begin{equation}\label{313}
\begin{split}
\eta \circ \left(\begin{array}{ccccc}
  D_{\tilde{Q}_{1}} &\cdots & \gamma_{1\tilde{l}} \\
   \vdots  &\ddots & \vdots\\
  0  & \cdots & D_{\tilde{Q}_{\tilde{l}}}
\end{array}\right)=\left(\begin{array}{ccccc}
  D_{Q_{1}} &\cdots & \beta_{1l} \\
   \vdots  &\ddots & \vdots\\
  0  & \cdots & D_{Q_{l}}
\end{array}\right)\circ \eta .
\end{split}
\end{equation}

Assume that $\eta_{l1}=0$, $\cdots $, $\eta_{(k+1)1}=0$ and $\eta_{k1}\neq 0$.
We express $\eta $ as a partitioned matrix
\begin{equation}
\eta=\left(\begin{array}{ccccc}
  \check{\eta}_{11} & & \check{\eta}_{12} \\
   \eta_{k1}  & & \check{\eta}_{22}\\
  0  &  & \check{\eta}_{32}
\end{array}\right),
\end{equation}
where $\check{\eta}_{11}=\left(\begin{array}{ccccc}
  \eta_{11}  \\
   \vdots  \\
  \eta_{(k-1)1}
\end{array}\right)$, $\check{\eta}_{12}=\left(\begin{array}{ccccc}
  \eta_{12} &\cdots & \eta_{1\tilde{l}} \\
   \vdots  &\ddots & \vdots\\
  \eta_{(k-1)2} & \cdots & \eta_{(k-1)\tilde{l}}
\end{array}\right)$, $\check{\eta}_{22}=\left(\begin{array}{ccccc}
  \eta_{k2} &\cdots & \eta_{k\tilde{l}} \\
  \end{array}\right)$, $\check{\eta}_{32}=\left(\begin{array}{ccccc}
  \eta_{12} &\cdots & \eta_{1\tilde{l}} \\
   \vdots  &\ddots & \vdots\\
  \eta_{(k-1)2} & \cdots & \eta_{(k-1)\tilde{l}}
\end{array}\right)$.
Write:
\begin{equation}\label{314}
\begin{split}
 \left(\begin{array}{ccccc}
  D_{\tilde{Q}_{1}} &\cdots & \gamma_{1\tilde{l}} \\
   \vdots  &\ddots & \vdots\\
  0  & \cdots & D_{\tilde{Q}_{\tilde{l}}}
\end{array}\right)=\left(\begin{array}{ccccc}
  D_{\tilde{Q}_{1}} & & \check{\gamma}_{12} \\
     0  &  & \check{D}_{\tilde{Q}_{2\tilde{l}}}
\end{array}\right)
\end{split}
\end{equation}
and
\begin{equation}\label{315}
\begin{split}
\left(\begin{array}{ccccc}
  D_{Q_{1}} &\cdots & \beta_{1l} \\
   \vdots  &\ddots & \vdots\\
  0  & \cdots & D_{Q_{l}}
\end{array}\right)=\left(\begin{array}{ccccc}
  \check{D}_{Q_{1(k-1)}} &\check{\beta }_{12} & \check{\beta }_{13} \\
   0  &D_{Q_{k}} & \check{\beta }_{23}\\
  0  & 0 & \check{D}_{Q_{(k+1)l}}
\end{array}\right),
\end{split}
\end{equation}
where $\check{\gamma }_{12}=\left(\begin{array}{ccccc}
  \gamma_{12} &\cdots & \gamma_{1\tilde{l}} \\
   \end{array}\right)$, $\check{D}_{\tilde{Q}_{2\tilde{l}}}=\left(\begin{array}{ccccc}
  D_{\tilde{Q}_{2}} &\cdots & \gamma_{2\tilde{l}} \\
   \vdots  &\ddots & \vdots\\
  0  & \cdots & D_{\tilde{Q}_{\tilde{l}}}
\end{array}\right)$, $\check{D}_{Q_{1(k-1)}}=\left(\begin{array}{ccccc}
  D_{Q_{1}} &\cdots & \beta_{1(k-1)} \\
   \vdots  &\ddots & \vdots\\
  0  & \cdots & D_{Q_{k-1}}
\end{array}\right)$, $\check{\beta}_{12}=\left(\begin{array}{ccccc}
  \beta_{1k} \\
   \vdots  \\
  \beta_{(k-1)k}
\end{array}\right)$, $\check{\beta}_{13}=\left(\begin{array}{ccccc}
  \beta_{1(k+1)} &\cdots & \beta_{1l} \\
   \vdots  &\ddots & \vdots\\
  \beta_{(k-1)(k+1)}  & \cdots & \beta_{(k-1)l}
\end{array}\right)$, $\check{\beta}_{23}=\left(\begin{array}{ccccc}
  \beta_{k(k+1)} &\cdots & \beta_{kl}
\end{array}\right)$, $\check{D}_{Q_{(k+1)l}}=\left(\begin{array}{ccccc}
  D_{Q_{k+1}} &\cdots & \beta_{(k+1)l} \\
   \vdots  &\ddots & \vdots\\
  0  & \cdots & D_{Q_{l}}
\end{array}\right)$.
(\ref{313}) tells us that
\begin{equation}\label{JH4}
\eta_{k1}\circ D_{\tilde{Q}_{1}}=D_{Q_{k}}\circ \eta_{k1}.
\end{equation}
Since $(\tilde{Q}_{1}, D_{\tilde{Q}_{1}})$ and $(Q_{k}, D_{Q_{k}})$ are simple,  $\eta_{k1}:(\tilde{Q}_{1}, D_{\tilde{Q}_{1}}) \rightarrow (Q_{k}, D_{Q_{k}}) $ is an isomorphism. Denote
\begin{equation}
\begin{split}
A:=\left(\begin{array}{ccccc}
  \Id_{Q_{1}\oplus \cdots \oplus Q_{k-1}} &-\check{\eta }_{11}\circ (\eta_{k1})^{-1} & 0 \\
   0  &\Id_{Q_{k}} & 0\\
  0  & 0 & \Id_{Q_{k+1}\oplus \cdots \oplus Q_{l}}
\end{array}\right)
\end{split}
\end{equation}
and
\begin{equation}
\begin{split}
B:=\left(\begin{array}{ccccc}
  \Id_{\tilde{Q}_{1}} &-(\eta_{k1})^{-1} \circ \check{\eta }_{22}  \\
    0  & & \Id_{\tilde{Q}_{2}\oplus \cdots \oplus \tilde{Q}_{\tilde{l}}}
\end{array}\right).
\end{split}
\end{equation}
By direct calculation, we have:
\begin{equation}\label{JH2}
A\circ \eta \circ B=\left(\begin{array}{ccccc}
  0 & & \check{\eta}_{12}- \check{\eta }_{11}\circ (\eta_{k1})^{-1} \check{\eta }_{22}\\
   \eta_{k1}  & & 0\\
  0  &  & \check{\eta}_{32}
\end{array}\right).
\end{equation}
(\ref{313}) is equivalent to the following formula
\begin{equation}\label{3131}
\begin{split}
A\circ \eta \circ B \circ B^{-1} \circ \left(\begin{array}{ccccc}
  D_{\tilde{Q}_{1}} &\cdots & \gamma_{1\tilde{l}} \\
   \vdots  &\ddots & \vdots\\
  0  & \cdots & D_{\tilde{Q}_{\tilde{l}}}
\end{array}\right)=A\circ \left(\begin{array}{ccccc}
  D_{Q_{1}} &\cdots & \beta_{1l} \\
   \vdots  &\ddots & \vdots\\
  0  & \cdots & D_{Q_{l}}
\end{array}\right)A^{-1} \circ A \circ \eta \circ B,
\end{split}
\end{equation}
and then
\begin{equation}\label{JH3}
\begin{split}
& \left(\begin{array}{ccccc}
  0 & & (\check{\eta}_{12}- \check{\eta }_{11}\circ (\eta_{k1})^{-1} \check{\eta }_{22})\circ \check{D}_{\tilde{Q}_{2\tilde{l}}}\\
   \eta_{k1}\circ D_{\tilde{Q}_{1}}  & & \check{\eta }_{22} \circ \check{D}_{\tilde{Q}_{2\tilde{l}}} -D_{Q_{k}}\circ \check{\eta }_{22} +\eta_{k1} \circ \check{\gamma }_{12} \\
  0  &  & \check{\eta}_{32}\circ \check{D}_{\tilde{Q}_{2\tilde{l}}}
\end{array}\right)\\=&
\left(\begin{array}{ccccc}
 \check{D}_{Q_{1(k-1)}}\circ  \check{\eta}_{11}- \check{\eta }_{11}\circ D_{\tilde{Q_{1}}}+ \check{\beta }_{12}\circ \eta_{k1} & & \check{D}_{Q_{1(k-1)}}\circ (\check{\eta}_{12}- \check{\eta }_{11}\circ (\eta_{k1})^{-1} \check{\eta }_{22})+ \check{\alpha }\circ \check{\eta}_{32}\\
 D_{Q_{k}}\circ  \eta_{k1}  & & \check{\beta }_{23} \circ  \check{\eta }_{32} \\
  0  &  & \check{D}_{Q_{(k+1)l}}\circ \check{\eta}_{32}
\end{array}\right),\\
\end{split}
\end{equation}
where $\check{\alpha }=\check{\beta}_{13}- \check{\eta }_{11}\circ (\eta_{k1})^{-1} \check{\beta }_{22}$. Set
\begin{equation}
\check{\eta}=\left(\begin{array}{ccccc}
   \check{\eta}_{12}- \check{\eta }_{11}\circ (\eta_{k1})^{-1} \check{\eta }_{22} \\
   \check{\eta}_{32}
\end{array}\right).
\end{equation}
From (\ref{JH2}) and (\ref{JH3}), it is not hard to see that $\check{\eta }:\tilde{Q}_{2}\oplus \cdots \oplus \tilde{Q}_{\tilde{l}}\rightarrow Q_{1}\oplus \cdots \oplus \tilde{Q}_{k-1}\oplus Q_{k+1}\oplus \cdots \oplus \tilde{Q}_{l}$ is a bundle isomorphism and satisfies
\begin{equation}\label{JH5}
\check{\eta} \circ \check{D}_{\tilde{Q}_{2\tilde{l}}} =\left(\begin{array}{ccccc}
 \check{D}_{Q_{1(k-1)}}& & \check{\alpha}\\
   0  &  & \check{D}_{Q_{(k+1)l}}
\end{array}\right)\circ \check{\eta} .
\end{equation}
According to (\ref{JH4}), (\ref{JH5}) and induction, we can prove there must exist an isomorphism between $\oplus_{i=1}^{l}(Q_{i}, D_{Q_{i}})$ and $\oplus_{\alpha =1}^{\tilde{l}}(\tilde{Q}_{\alpha }, D_{\tilde{Q}_{\alpha }})$.
\end{proof}

\begin{prop}\label{t1}
Let $(\hat{E}, \hat{D})$ be a flat complex vector bundle over a compact Riemannian manifold $(M, g)$,  $\hat{K}$ be a Hermitian metric on $\hat{E}$ and  $i_{0}:\hat{S}\hookrightarrow \hat{E}$ be a $\hat{D}$-invariant sub-bundle of $\hat{E}$. Suppose that there is a sequence of gauge transformation $\hat{\sigma}_{l}$ such that $\hat{D}_{l}:=\hat{\sigma}_{l}(\hat{D})\rightarrow \hat{D}_{\infty}$ weakly in $L_{1}^{p}$-topology as $l\rightarrow +\infty$. Furthermore, $\|\hat{D}_{l, K}^{\ast }\psi_{\hat{D}_{l}, K}\|_{L^{\infty}}$ and $\|\psi_{\hat{D}_{l}, K}\|_{L^{\infty}}$ are uniformly bounded. Then there is a subsequence of $\eta_{l}:=\hat{\sigma}_{l}\circ i_{0}$, up to rescale, converges weakly to a nonzero map $\eta_{\infty}:\hat{S}\rightarrow \hat{E}$ satisfying $\eta_{\infty}\circ D_{\hat{S}}=\hat{D}_{\infty}\circ \eta _{\infty}$ in $L_{2}^{p}$-topology, where $D_{\hat{S}}=\hat{D}|_{\hat{S}}$ is the induced flat connection on $\hat{S}$.
\end{prop}
\begin{proof}
 With respect to the Hermitian metric $K$, we have the following decomposition
\begin{equation}
\hat{D}_{l}=\hat{D}_{l, K}+\psi_{\hat{D}_{l}, K},
\end{equation}
where $\hat{D}_{l, K}$ is a $K$-unitary connection, $\psi_{\hat{D}_{l}, K}\in \Omega^{1}(\mbox{End}(E))$ is $K$-self-adjoint. Under the condition  $\hat{D}_{l}\rightarrow \hat{D}_{\infty}$ weakly in $L_{1}^{p}$-topology, we know that
\begin{equation}
\hat{D}_{l, K}\rightarrow \hat{D}_{\infty , K}, \quad and \quad \psi_{\hat{D}_{l}, K}\rightarrow \psi_{\hat{D}_{\infty}, K}
\end{equation}
weakly in $L_{1}^{p}$-topology as $l\rightarrow \infty $. Choose local  coordinates $\{x^{i}\}_{i=1}^{n}$ on $M$, and write $g=g_{ij}dx^{i}\otimes dx^{j}$. Note that $\hat{S}$ is a $\hat{D}$-invariant sub-bundle, and we have
\begin{equation}
\eta_{l}\circ D_{\hat{S}} =\hat{D}_{l} \circ \eta_{l},
\end{equation}
\begin{equation}\label{E01}
\begin{split}
\Delta_{K, l} \eta_{l}&=g^{ij}(\nabla^{K, l}_{\frac{\partial }{\partial x^{i}}}(\hat{D}_{l, K}\circ \eta_{l}-\eta_{l}\circ D_{\hat{S}, K}))(\frac{\partial }{\partial x^{j}})\\
&=\hat{D}_{l, K}^{\ast }\psi_{\hat{D}_{l}, K} \circ \eta_{l}-\eta_{l}\circ D_{\hat{S}, K}^{\ast }\psi_{D_{\hat{S}}, K} -2 g^{ij} \psi_{\hat{D}_{l}, K}(\frac{\partial }{\partial x^{i}}) \circ  \eta_{l}\circ \psi_{D_{\hat{S}}, K}(\frac{\partial }{\partial x^{j}})\\
&+g^{ij} \psi_{\hat{D}_{l}, K}(\frac{\partial }{\partial x^{i}}) \circ \psi_{\hat{D}_{l}, K}(\frac{\partial }{\partial x^{j}}) \circ  \eta_{l}+ g^{ij}\eta_{l}\circ \psi_{D_{\hat{S}}, K}(\frac{\partial }{\partial x^{i}})\circ \psi_{D_{\hat{S}}, K}(\frac{\partial }{\partial x^{j}})
\end{split}
\end{equation}
and
\begin{equation}\label{EC0}
\Delta |\eta_{l}|_{K}^{2}\geq -\hat{C}_{0}(|\hat{D}_{l, K}^{\ast }\psi_{\hat{D}_{l}, K}|_{K}+|D_{\hat{S}, K}^{\ast }\psi_{D_{\hat{S}}, K}|_{K}+|\psi_{\hat{D}_{l}, K}|_{K}^{2}+|\psi_{D_{\hat{S}}, K}|_{K}^{2})|\eta_{l}|_{K}^{2},
\end{equation}
where $\hat{C}_{0}$ is a constant depending only on the dimension of $M$. Set $\tilde{\eta}_{l}=\frac{\eta_{l}}{\|\eta_{l}\|_{L^{2}}}$. From (\ref{EC0}) and the Moser's iteration, we see that there exists a uniform constant $\hat{C}_{1}$ such that
\begin{equation}\label{EC01}
\|\tilde{\eta}_{l}\|_{L^{\infty}}\leq \hat{C}_{1}
\end{equation}
for all $l$.
By the above uniform $C^{0}$-estimate (\ref{EC01}), the equation (\ref{E01}) and  the assumption that  $\hat{D}_{l}\rightarrow \hat{D}_{\infty}$ weakly in $L_{1}^{p}$-topology, the elliptic theory gives us that there exists a subsequence of $\tilde{\eta}_{l}$  which converges weakly in $L_{2}^{p}$-topology to a  map $\eta_{\infty}$ such that $\eta_{\infty}\circ D_{\hat{S}}=\hat{D}_{\infty}\circ \eta _{\infty}$. On the other hand, the fact that $\|\tilde{\eta}_{l}\|_{L^{2}}=1$ for all $l$ implies that the map $\eta_{\infty}$ is non-zero.
\end{proof}

\begin{theorem}\label{a1}
Let $(E, D)$ be a rank $r$ flat complex vector bundle over a compact Riemannian manifold $(M, g)$,  $K$ be a Hermitian metric on $E$. Suppose that there is a sequence of gauge transformation $\sigma_{j}$ such that $D_{j}:=\sigma_{j}(D)\rightarrow D_{\infty}$ weakly in $L_{1}^{p}$-topology  with $D_{\infty, K}^{\ast }\psi_{D_{\infty }, K}=0$. Furthermore, $\|D_{j, K}^{\ast }\psi_{D_{j}, K}\|_{L^{\infty}}$ and $\|\psi_{D_{j}, K}\|_{L^{\infty}}$ are uniformly bounded. Then, we have:
 \begin{equation}\label{I0}
(E, D_{\infty })\cong Gr^{JH}(E, D),
\end{equation}
where $ Gr^{JH}(E, D)$ is the graded  flat complex vector bundle associated to the Jordan-H\"older filtration of $(E, D)$.
\end{theorem}
\begin{proof}
We  prove this by  induction. Let's assume that the conclusion of this theorem is  true for $\rank(E)<r$. If $(E, D)$ is simple, Proposition \ref{t1}
implies that there exists an isomorphic map between $(E, D)$ and $(E, D_{\infty })$. Suppose $(E, D)$ is not simple, and then we have the following Jordan-H\"older filtration of sub-bundles
\begin{equation}\label{HNS010}
0=E_{0}\subset E_{1}\subset \cdots \subset E_{i} \cdots \subset E_{l}=E ,
\end{equation}
such that every sub-bundle $E_{i}$ is $D$-invariant and every quotient bundle $(Q_{i}, D_{i}):=(E_{i}/E_{i-1}, D_{i})$ is flat and simple. Let $S=E_{1}$ and $Q=E/E_{1}$. We consider the following exact sequence
\begin{equation}
0\rightarrow S\xrightarrow{i_{0}} E \xrightarrow{P} Q \rightarrow 0.
\end{equation}
Denote that $D_{S}$ and $D_{Q}$ are the induced connections on $S$ and $Q$, $H_{j}=K\sigma_{j}^{\ast}\circ \sigma_{j}$, $\pi_{1}^{H_{j}}$ is the orthogonal projection onto $E_{1}$ with respect to the metric $H_{j}$. Set $\pi_{1}^{j}=\sigma_{j}\circ \pi_{1}^{H_{j}} \circ \sigma_{j}^{-1}$.
It is easy to see that \begin{equation}(\pi_{1}^{j})^{\ast K}=\pi_{1}^{j}=(\pi_{1}^{j})^{2}\end{equation}
and
\begin{equation}
(\Id_{E}-\pi_{1}^{j})\circ D_{j}\pi_{1}^{j}=0.
\end{equation}
According to (\ref{pi2}), (\ref{pi3}), (\ref{codazzi1}), (\ref{C120}) and the conditions of the theorem, we derive
\begin{equation}\label{pi4}
\begin{split}
\int_{M} |D_{j}\pi_{1}^{j}|_{K}^{2}dV_{g}&=\int_{M}|\pi_{1}^{j}\circ D_{j}\pi_{1}^{j}|_{K}^{2}dV_{g}=\int_{M}|\pi_{1}^{H_{j}}\circ D\pi_{1}^{H_{j}}|_{H_{j}}^{2}dV_{g}\\
&=2\int_{M}\langle D_{S, H_{j}}^{\ast}\psi_{D_{S}, H_{j}}-\pi_{1}^{H_{j}}\circ D_{H_{j}}^{\ast }\psi_{D, H_{j}}\circ i_{0}, \Id_{S}\rangle_{H_{j}}dV_{g}\\&=-2\int_{M}\langle\pi_{1}^{H_{j}}\circ D_{H_{j}}^{\ast }\psi_{D, H_{j}}\circ i_{0}, \Id_{S}\rangle_{H_{j}}dV_{g}\\
&\leq 2\int_{M}|D_{H_{j}}^{\ast }\psi_{D, H_{j}}|_{H_{j}}dV_{g}\\
&=2\int_{M}|D_{j}^{\ast }\psi_{D_{j}, K}|_{K}dV_{g}\rightarrow 0.
\end{split}
\end{equation}
On the other hand, $|\pi_{1}^{j}|_{K}\equiv \rank (S)$. After going to a subsequence, one can obtain $\pi_{1}^{j} \rightarrow \pi_{1}^{\infty }$ strongly in $L^{p}\cap L_{1}^{2}$ , and \begin{equation}D_{\infty}\pi_{1}^{\infty}=0.\end{equation}  We know that $\pi_{1}^{\infty}$ determines a $D_{\infty}$-invariant sub-bundle $E_{1}^{\infty}$ of $(E, D_{\infty})$ with $\rank (E_{1}^{\infty})=\rank (E_{1})$, and
\begin{equation}
(E, D_{\infty})\cong (E_{1}^{\infty}, D_{1, \infty})\oplus (Q_{\infty}, D_{Q_{\infty}}),
\end{equation}
where $Q_{\infty}=(E_{1}^{\infty})^{\bot K}$, $D_{1, \infty}$ and $D_{Q_{\infty}}$ are the induced connections on $E_{1}^{\infty}$ and $Q_{\infty}$ by the connection $D_{\infty}$.

Proposition \ref{t1} yields that there is a subsequence of $\eta_{j}:=\frac{\sigma_{j}\circ i_{0}}{\|\sigma_{j}\circ i_{0}\|_{L^{2}}}$, up to rescale, converges to a nonzero map $\eta_{\infty}:S\rightarrow E$ satisfying $\eta_{\infty}\circ D_{S}=D_{\infty}\circ \eta _{\infty}$.
Due to $\pi_{1}^{j}\circ \sigma_{j}\circ i_{0}=\sigma_{j}\circ i_{0}$,  we have: \begin{equation}\pi_{1}^{\infty }\circ \eta_{\infty }=\eta_{\infty }.\end{equation}
The condition $D_{\infty, K}^{\ast }\psi_{D_{\infty }, K}=0$ implies that $D_{\infty}$ is smooth, and then $\eta _{\infty}$ is also smooth. Because $E_{1}$ is simple, it is easy to see that $\eta_{\infty}$ is an isomorphic map between $(E_{1}, D_{S})$ and $(E_{1}^{\infty}, D_{1, \infty })$.

Let $\{e_{\alpha }\}$ be a local frame of $E_{1}$, and $H_{j, \alpha \beta}=\langle \eta_{j}(e_{\alpha }), \eta_{j}(e_{\beta})\rangle_{K}$. We write
\begin{equation}
\pi_{1}^{j} (Y)=\langle Y, \eta_{j}(e_{\beta })\rangle_{K}H_{j}^{\alpha \beta }\eta_{j}(e_{\alpha})
\end{equation}
for any $Y \in \Gamma(E)$, where $(H_{j}^{\alpha \beta })$ is the inverse of the matrix $(H_{j, \alpha \beta})$.
Since $\eta_{j}\rightarrow
\eta_{\infty}$ weakly in $L_{2}^{p}$-topology, and $\eta_{\infty}$ is injective, we know  that $\pi_{1}^{j}\rightarrow
\pi_{1}^{\infty}$ weakly in $L_{2}^{p}$-topology . Here, $\pi_{1}^{\infty}: E \rightarrow E$ is just  the projection onto $E_{1}^{\infty}$ with respect to the metric $K$.

Using
 Lemma 5.12 in \cite{Da}, we can choose a sequence of $K$-unitary gauge
transformations $u_{j}$ such that
$\pi_{1}^{j}=u_{j}\circ \pi_{1}^{\infty}\circ u_{j}^{-1}$  and $u_{j}\rightarrow \Id_{E}$
weakly in $L_{2}^{p}$-topology as $j\rightarrow \infty$.
It is straightforward to check that $u_{j}(Q_{\infty})=u_{j}((E_{1}^{\infty})^{\bot K})=(\pi_{1}^{j} (E))^{\bot K}$, and the $K$-unitary gauge
transformation $u_{0}$ satisfies $u_{0}((E_{1}^{\infty})^{\bot K})=E_{1}^{\bot K}$.
Set
\begin{equation}
D_{j}^{Q}=(P^{\ast K})^{-1}\circ u_{0} \circ (\pi_{1}^{\infty})^{\bot K} \circ u_{j}^{-1} \circ D_{j} \circ u_{j} \circ (\pi_{1}^{\infty})^{\bot K} \circ u_{0}^{-1}\circ P^{\ast K} ,
\end{equation}
\begin{equation}
\hat{\sigma }_{j}=(P^{\ast K})^{-1}\circ u_{0} \circ (\pi_{1}^{\infty})^{\bot K} \circ u_{j}^{-1} \circ \sigma_{j} \circ P^{\ast K}
\end{equation}
and
\begin{equation}
\hat{\sigma }_{j}^{-1}=(P^{\ast K})^{-1} \circ (\pi_{1}^{\infty})^{\bot K} \circ \sigma_{j}^{-1} \circ u_{j} \circ u_{0} \circ P^{\ast K}.
\end{equation}
One can find that
\begin{equation}
D_{j}^{Q}=\hat{\sigma }_{j}\circ D_{Q} \circ \hat{\sigma }_{j}^{-1}
\end{equation}
and
\begin{equation}
D_{j}^{Q}\rightarrow D_{\infty}^{Q}=(P^{\ast K})^{-1}\circ u_{0} \circ D_{Q_{\infty}}\circ u_{0}^{-1}\circ P^{\ast K}\end{equation}
weakly in $L_{1}^{p}$-topology.
 From (\ref{Q9}), $(\ref{codazzi1})$ and (\ref{pi4}), it follows that $\|\psi_{D^{Q}_{j}, K}\|_{L^{\infty}}$ and $\|(D^{Q}_{j, K})^{\ast }\psi_{D^{Q}_{j}, K}\|_{L^{\infty}}$  are uniformly bounded, and $D_{Q_{\infty}, K}^{\ast }\psi_{D_{Q_{\infty}}, K}=0$.
 According to the induction hypothesis, we have
  \begin{equation}\label{I0}
(Q_{\infty}, D_{Q_{\infty} })\cong (Q, D_{\infty}^{Q})\cong Gr^{JH}(Q, D_{Q}).
\end{equation}
 This completes the proof of theorem.

\end{proof}

\begin{proof}[Proof of Theorem \ref{thm0}]
Thanks to Proposition \ref{b12} (\cite{Cor}), we know that the harmonic flow (\ref{H1}) has a long time solution $\sigma (t)$ for $t\in [0, \infty )$, and there exists a sequence $t_{j}\rightarrow \infty $  such that $\sigma(t_{j}) \{D\}$ converges weakly,  modulo $K$-unitary gauge transformations, to a flat connection $D_{\infty}$ in $L_{1}^{p}$-topology,  and $D_{\infty ,K}^{\ast}\psi_{\infty , K}=0$. (\ref{C15}) and (\ref{C13}) imply that $\|\psi_{\sigma (t) \{D\}, K}\|_{L^{\infty }}$ and $\|(\sigma(t) \{D\})_{K}^{\ast }\psi_{\sigma (t) \{D\}, K}\|_{L^{\infty}}$ are uniformly bounded.  By  Theorem \ref{a1}, we deduce
\begin{equation}
(E, D_{\infty })\cong Gr^{JH}(E, D).
\end{equation}
\end{proof}

\end{document}